\documentclass[preprint,11pt]{elsarticle}
\usepackage[T1]{fontenc}
\usepackage{lmodern}
\usepackage{microtype}
\usepackage{amsmath,amssymb,amsthm,mathtools}
\usepackage[colorlinks=true,linkcolor=blue,citecolor=blue,urlcolor=blue]{hyperref}

\newtheorem{theorem}{Theorem}[section]
\newtheorem{lemma}[theorem]{Lemma}
\newtheorem{proposition}[theorem]{Proposition}
\newtheorem{corollary}[theorem]{Corollary}
\theoremstyle{remark}

\newcommand{\E}{\mathbb E}

\hypersetup{
 pdftitle={Inequalities for Higher-Order Logarithmic Differences of Boros--Moll Coefficients},
 pdfauthor={Cindy C. Y. Gu, Yafeng Liu, and Matthew H. Y. Xie},
 pdfkeywords={Boros--Moll coefficients, logarithmic differences, moment sequences, total positivity}
}

\begin{document}

\begin{frontmatter}
\title{Inequalities for Higher-Order Logarithmic Differences of Boros--Moll Coefficients}
\author[tjut]{Cindy C.~Y. Gu}
\ead{guchunyan\_math@email.tjut.edu.cn}
\author[tjut]{Yafeng Liu}
\ead{liuyafeng@stud.tjut.edu.cn}
\author[tjut]{Matthew H.~Y. Xie\corref{cor1}}
\ead{xie@email.tjut.edu.cn}
\cortext[cor1]{Corresponding author.}
\address[tjut]{School of Mathematical Sciences, Tianjin University of Technology,
Tianjin 300384, P. R. China}

\begin{abstract}
Let $d_i(m)$ denote the coefficient of $x^i$ in the Boros--Moll polynomial $P_m(x)$.
We prove that the sequence $\{d_{i-1}(m)d_{i+1}(m)/d_i(m)^2\}_{i=1}^{m-1}$ is strictly log-concave for every $m\ge4$, settling a conjecture of Chen and Gu. For every $m\ge 3$, we show that the consecutive ratios $d_i(m)/d_{i-1}(m)$ form a strictly log-convex sequence for $1\le i\le\lfloor m/2\rfloor+1$ and a strictly log-concave sequence for $\lfloor m/2\rfloor\le i\le m$, as conjectured by Zhao.
For each fixed $i\ge 1$, we establish $(-1)^{r-1}\Delta_m^r\log d_i(m)>0$ for $r\ge 2$ and $m\ge i$, where $\Delta_m$ denotes the forward difference with respect to $m$. For $i=0$, we obtain the opposite inequalities. We thereby prove a conjecture of Jiang and Wang and two further conjectures of Zhao. For every integer $i\ge 0$, we prove strict infinite log-monotonicity of the sequence $\{d_i(m-1)d_i(m+1)/(\binom{m}{i+1}d_i(m)^2)\}_{m\ge i+1}$, strengthening another conjecture of Zhao. We further show that every positive real power of the binomially normalized ratio sequence is a Hausdorff moment sequence with a strictly totally positive Hankel matrix. We also obtain asymptotic formulas for forward differences of arbitrary fixed order, resolving a conjecture of Jiang and Wang.
For integers $i\ge 1$, we prove that the sequence $\{2^n d_i(i)/d_i(i+n)\}_{n\ge 0}$ is an infinitely divisible Hausdorff moment sequence and that the Hankel matrix $[1/d_i(i+j+k)^\tau]_{j,k\ge 0}$ is strictly totally positive for every $\tau>0$. For the integral interpolation, we show that the Hausdorff moment property of the normalized reciprocals holds precisely when the positive real coefficient index is at least $1/4$.
\end{abstract}
\end{frontmatter}

\section{Introduction}\label{sec:introduction}

This paper is concerned with higher-order logarithmic differences of the Boros--Moll coefficients $d_i(m)$ in each variable separately.

A positive sequence $\{a_k\}$ is said to be log-concave if $a_k^2 \ge a_{k-1}a_{k+1}$ holds at every interior index, and log-convex if all inequalities are reversed. It is strictly log-concave (or strictly log-convex) if all inequalities are strict. Furthermore, log-concavity of $\{a_k\}$ is equivalent to the consecutive ratios $a_{k+1}/a_k$ being nonincreasing, and whether this ratio sequence is log-convex or log-concave is determined by the third differences of $\log a_k$.

For a nonnegative integer $m$ and a real number $x>-1$, Boros and Moll \cite{BorosMoll} evaluated the quartic integral as
\[
\int_0^\infty \frac{1}{(t^4+2xt^2+1)^{m+1}}\,dt = \frac{\pi}{2^{m+3/2}(x+1)^{m+1/2}}P_m(x),
\]
where
\[
P_m(x)=\sum_{i=0}^{m}d_i(m)x^i.
\]
The coefficients $d_i(m)$ are given explicitly by
\begin{equation}\label{eq:coefficients}
 d_i(m) = 4^{-m}\sum_{k=i}^{m}2^k\binom{2m-2k}{m-k}\binom{m+k}{k}\binom{k}{i},
 \qquad m\ge 0,\quad 0\le i\le m.
\end{equation}

Boros and Moll \cite{BorosMollUnimodal} proved that $\{d_i(m)\}_{i=0}^m$ is unimodal. Kauers and Paule \cite{KauersPaule} proved its log-concavity by establishing recurrence relations for the coefficients. Chen, Pang, and Qu \cite{ChenPangQu} constructed a weighted combinatorial model for a normalization of these coefficients using partially 2-colored permutations, and used this model to explain the log-concavity of $\{d_i(m)\}_{i=0}^m$. Chen and Gu \cite{ChenGu} proved reverse ultra log-concavity, and Chen and Xia \cite{ChenXia} proved ratio monotonicity. Chen, Wang, and Xia \cite{ChenWangXiaInterlacing} showed that the consecutive coefficient ratios of $P_m$ and $P_{m+1}$ interlace.

To describe higher-order log-concavity, define the operator
\[
 \mathcal L(a)_k=a_k^2-a_{k-1}a_{k+1},
\]
with a finite sequence extended by zero. A sequence is $r$-log-concave if its first $r$ iterates under $\mathcal L$ are nonnegative, and infinitely log-concave if all its iterates are nonnegative.

Chen and Xia \cite[Theorem~1.4]{ChenXia2} proved 2-log-concavity of the Boros--Moll coefficients for $m\ge 2$. Chen, Dou, and Yang \cite{ChenDouYang} proved 3-log-concavity by establishing Br\"and\'en's real-rootedness conjecture for
\[
 \sum_{i=0}^m\frac{d_i(m)}{(i+2)!}z^i.
\]
Xie and Zhang \cite[Theorem~1.1 and Corollary~1.2]{XieZhangRows} proved that the generating polynomial obtained after one application of $\mathcal L$ has only real zeros. Infinite log-concavity then follows from Br\"and\'en's preservation theorem \cite{Branden}. Xie, Zhang, and Zhang \cite[Theorem~3.1]{XieZhangUnified} gave another proof using analytic estimates and finite verification. Guo \cite{GuoTuran} proved higher order Tur\'an inequalities involving four consecutive coefficients, and Zhao \cite{ZhaoTuran} proved a strict form of these inequalities.

The known infinite log-concavity results concern successive applications of $\mathcal L$ to the coefficient sequence. They do not settle the conjectures considered here, which concern the log-concavity and log-convexity of ratios of the original coefficients. These questions require estimates for higher differences of $\log d_i(m)$. We obtain these estimates from the integral and Laplace representations in \cite{XieZhangRows,XieZhangUnified}.

We first consider the ratios $d_{i-1}(m)d_{i+1}(m)/d_i(m)^2$. Since their logarithms are the second differences of $\log d_i(m)$, log-concavity of the ratio sequence is equivalent to nonpositivity of the fourth differences with respect to $i$. Chen and Gu \cite[Conjecture~4.3]{ChenGu} conjectured that the ratio sequence indexed by $2\le i\le m-2$ is log-concave. Aravinda \cite[Theorem~1.4]{Aravinda} proved strict log-concavity for all sufficiently large $m$, without giving an explicit threshold. We prove the following strict form of the conjecture for every $m\ge 4$.

\begin{theorem}[{\cite[Conjecture~4.3]{ChenGu}}]\label{thm:horizontal}
For every integer $m\ge 4$, let
\begin{equation}\label{eq:ratios}
 u_i(m)=\frac{d_{i-1}(m)d_{i+1}(m)}{d_i(m)^2},\qquad 1\le i\le m-1.
\end{equation}
Then
\begin{equation}\label{eq:strict-lc}
 u_i(m)^2>u_{i-1}(m)u_{i+1}(m),\qquad 2\le i\le m-2.
\end{equation}
\end{theorem}

Theorem~\ref{thm:horizontal} includes the inequalities at $i=2$ and $i=m-2$, in addition to those at $3\le i\le m-3$ required by the original conjecture. The proof shows that a smooth interpolation of $\log d_i(m)$ has a negative fourth derivative throughout $[0,m]$. This approach also yields strict log-concavity for the interpolated ratios at any step size $h$ with $0<h<m/2$.

Turning to the consecutive ratios, we define
\begin{equation}\label{eq:consecutive}
 q_i(m)=\frac{d_i(m)}{d_{i-1}(m)},\qquad 1\le i\le m.
\end{equation}
They are related to $u_i(m)$ by
\[
 \log\frac{q_{i-1}(m)q_{i+1}(m)}{q_i(m)^2}=\log\frac{u_i(m)}{u_{i-1}(m)},\qquad 2\le i\le m-1.
\]
By Theorem~\ref{thm:horizontal}, this expression is strictly decreasing in $i$ for $m\ge 4$, so its sign can change at most once. The following theorem determines where the sign changes.

\begin{theorem}[{\cite[Conjecture~4.4]{Zhao}}]\label{thm:half}
For every integer $m\ge 3$,
\begin{align}
 q_i(m)^2&<q_{i-1}(m)q_{i+1}(m), && 2\le i\le \lfloor m/2\rfloor,\label{eq:left-half}\\
 q_i(m)^2&>q_{i-1}(m)q_{i+1}(m), && \lfloor m/2\rfloor+1\le i\le m-1.\label{eq:right-half}
\end{align}
\end{theorem}

Now fix $\ell$ and consider $\{d_\ell(m)\}_{m\ge \ell}$. Jiang and Wang \cite{JiangWang} proved that $\{d_0(m)\}_{m\ge 0}$ is log-convex, whereas $\{d_\ell(m)\}_{m\ge \ell}$ is log-concave for each $\ell\ge 1$. In Conjecture~4.5, they proposed an extension of these inequalities to logarithmic differences of every order. Let
\[
 \Delta_m f(m)=f(m+1)-f(m)
\]
denote the forward difference with respect to $m$. Zhang and Zhao \cite[Theorem~3.1]{ZhangZhaoBriggs} proved that the consecutive ratios of $\{d_\ell(m)\}_{m\ge \ell}$ form a strictly log-convex sequence for every $\ell\ge 1$. In terms of logarithmic differences, this means
\[
 \Delta_m^3\log d_\ell(m)>0,\qquad \ell\ge 1,\quad m\ge \ell.
\]
The following theorem proves Jiang and Wang's Conjecture~4.5 for all orders.

\begin{theorem}[{\cite[Conjecture~4.5]{JiangWang}}]\label{thm:column-log-differences}
For integers $\ell\ge 0$ and $m\ge \ell$, one has
\[
 \Delta_m\log d_\ell(m)>0.
\]
For every integer $r\ge 2$, the higher differences satisfy
\begin{align}
 (-1)^r\Delta_m^r\log d_0(m)&>0 &&(r\ge 2,\ m\ge 0),\label{eq:zero-differences}\\
 (-1)^{r-1}\Delta_m^r\log d_\ell(m)&>0
 &&(r\ge 2,\ \ell\ge 1,\ m\ge \ell).\label{eq:positive-differences}
\end{align}
\end{theorem}

The proof relies on a Laplace representation whose kernel is positive for $\ell\ge 1$ and negative for $\ell=0$. Each successive difference contributes a factor of $1-e^{-t}$ and a change of sign, yielding the inequalities for every $r\ge 2$.

These inequalities extend to successive ratio sequences. For a positive sequence $a$, define the ratio operator by
\[
 (\mathcal R a)_m=\frac{a_{m+1}}{a_m}.
\]
Following Chen, Guo, and Wang \cite{ChenGuoWang}, we say that $a$ is infinitely log-monotonic if $\mathcal R^j a$ is log-convex for every even $j\ge 0$ and log-concave for every odd $j\ge 1$. We say it is strictly infinitely log-monotonic if all the corresponding inequalities are strict. Chen, Guo, and Wang \cite[Theorem~2.1]{ChenGuoWang} proved that a positive smooth function yields an infinitely log-monotonic sequence at the integers if the second derivative of its logarithm is completely monotone.

Using the identity
\[
 \Delta_m^2\log(\mathcal R^j a)_m=\Delta_m^{j+2}\log a_m
\]
and Theorem~\ref{thm:column-log-differences}, we obtain strict infinite log-monotonicity of $\{d_0(m)\}_{m\ge 0}$ and of $\{d_\ell(m+1)/d_\ell(m)\}_{m\ge \ell}$ for every $\ell\ge 1$. This proves Zhao's Conjecture~4.5 \cite{Zhao}.

For each integer $\ell\ge 0$, define
\begin{equation}\label{eq:transposed-ratios}
 c_m^{(\ell)}=\frac{d_\ell(m-1)d_\ell(m+1)}{d_\ell(m)^2}, \qquad m\ge \ell+1.
\end{equation}
Taking $r=4$ in Theorem~\ref{thm:column-log-differences} yields strict log-concavity of $\{c_m^{(\ell)}\}_{m\ge \ell+1}$ for every $\ell\ge 1$, confirming Zhao's Conjecture~4.2 \cite{Zhao}. Zhao also considered the binomial normalization
\begin{equation}\label{eq:binomial-normalized-sequence}
 b_m=\frac{c_m^{(\ell)}}{\binom m{\ell+1}}, \qquad m\ge \ell+1.
\end{equation}
Zhao \cite[Conjecture~4.3]{Zhao} conjectured that $\{b_m\}_{m\ge \ell+1}$ is log-convex. We prove the following strict inequality.

\begin{theorem}[{\cite[Conjecture~4.3]{Zhao}}]\label{thm:binomial-normalization}
For all integers $\ell\ge 0$ and $m\ge \ell+2$,
\begin{equation}\label{eq:43}
 \frac{(c_m^{(\ell)})^2}{c_{m-1}^{(\ell)}c_{m+1}^{(\ell)}}<\frac{m(m-\ell)}{(m+1)(m-\ell-1)}.
\end{equation}
\end{theorem}

This inequality is equivalent to strict log-convexity of $\{b_m\}_{m\ge \ell+1}$. We also prove that this sequence is strictly infinitely log-monotonic. The next theorem strengthens strict log-convexity to positivity of all Hankel minors. A matrix is strictly totally positive if every finite minor with arbitrary increasing row and column indices is positive.

\begin{theorem}\label{thm:binomial-total-positivity}
For every integer $\ell\ge 0$ and every real $\tau>0$, the Hankel matrix
\[
 \left[b_{\ell+1+i+j}^{\,\tau}\right]_{i,j\ge 0}
\]
is strictly totally positive.
\end{theorem}

For every fixed pair of nonnegative integers $\ell$ and $r$, we obtain
\[
 \lim_{m\to\infty}\frac{\Delta_m^r d_\ell(m)}{d_\ell(m)}=1.
\]
Since $d_\ell(m+1)/d_\ell(m)$ tends to $2$, we also have
\[
 \lim_{m\to\infty}\frac{\Delta_m^r d_\ell(m+1)}{\Delta_m^r d_\ell(m)}=2,
\]
which proves Jiang and Wang's Conjecture~4.4 \cite{JiangWang}.

Recall that a Hausdorff moment sequence has a representation
\[
 s_n=\int_0^1 t^n\,d\sigma(t),\qquad n\ge 0,
\]
where $\sigma$ is a finite positive measure on $[0,1]$. When $s_0=1$, the measure $\sigma$ is a probability measure. A probability Hausdorff moment sequence is infinitely divisible if $\{s_n^\tau\}_{n\ge 0}$ is again a probability Hausdorff moment sequence for every $\tau>0$ \cite[Sections~1 and~4]{BergDuran}.

We now turn to moment representations of the reciprocal coefficients. For every integer $\ell\ge 1$ and every $\tau>0$, we prove that the sequence $\{1/d_\ell(\ell+n)^\tau\}_{n\ge 0}$ is a Hausdorff moment sequence. Every Hausdorff moment sequence is a Stieltjes moment sequence and has a Hankel matrix with nonnegative minors \cite[Section~1]{WangZhu}.

Berg \cite[Theorem~6.2]{BergPowersII} established infinitely divisible Hausdorff moment representations for certain ratios of Gamma functions. To analyze the reciprocal coefficients, we extend the coefficient index to a real parameter using the generalized binomial coefficient
\[
 \binom{a}{b}=\frac{\Gamma(a+1)}{\Gamma(b+1)\Gamma(a-b+1)},\qquad a>-1,\quad b>-1,\quad a-b>-1.
\]
For $x\ge 0$ and $y>x-1/2$, define the positive integral interpolation
\begin{equation}\label{eq:bivariate-D}
 D(x,y)=\frac{2^{y+1}\Gamma(2y+3/2)}{\sqrt{\pi}\Gamma(y+1)^2}\binom{y}{x}\int_0^1\frac{t^{y+x}}{\sqrt{1-t}(1+t)^{2y+3/2}}\,dt.
\end{equation}
By the integral formula in \cite{XieZhangRows,XieZhangUnified}, we have $D(i,m)=d_i(m)$ for integers $0\le i\le m$. For real $\lambda>0$, let
\begin{equation}\label{eq:reciprocal-column}
 a_n(\lambda)=\frac{2^n D(\lambda,\lambda)}{D(\lambda,\lambda+n)},\qquad n\ge 0.
\end{equation}

A smooth function $f$ on an interval $I$ is completely monotone if
\[
 (-1)^j f^{(j)}(y)\ge 0,\qquad j\ge 0,\quad y\in I.
\]
If all inequalities are strict, $f$ is strictly completely monotone; see \cite[Chapter~1]{SchillingSongVondracek}. The next theorem determines the exact parameter range for the moment property and for complete monotonicity.

\begin{theorem}[Reciprocal moments and complete monotonicity]\label{thm:threshold-moments}
For every real $\lambda>0$, the sequence $\{a_n(\lambda)\}_{n\ge 0}$ is a probability Hausdorff moment sequence if and only if $\lambda\ge 1/4$. The function
\[
 -\frac{\partial^2}{\partial y^2}\log D(\lambda,y)
\]
is strictly completely monotone on $[\lambda,\infty)$ if and only if $\lambda\ge 1/4$. For every $\lambda\ge 1/4$ and $\tau>0$, there is a probability measure $\sigma_{\lambda,\tau}$ on $[0,1]$ with infinitely many support points in $(0,1]$ such that
\[
 a_n(\lambda)^\tau=\int_0^1 t^n\,d\sigma_{\lambda,\tau}(t),\qquad n\ge 0.
\]
In particular, $\{a_n(\lambda)\}_{n\ge 0}$ is infinitely divisible for every $\lambda\ge 1/4$.
\end{theorem}

It follows from Theorem~\ref{thm:threshold-moments} that, for every integer $\ell\ge 1$ and every $\tau>0$, the Hankel matrix
\[
 \left[\frac{1}{d_\ell(\ell+i+j)^\tau}\right]_{i,j\ge 0}
\]
is strictly totally positive.

The integral interpolation comes from \cite{XieZhangRows,XieZhangUnified}. Its factorization, the positive Laplace representation for the negative logarithm of the weighted average, and the accompanying density bound were established in \cite{XieZhangUnified}. We also use their kernel for normalized ratios. For fixed degree, we prove sharp bounds for the third and fourth logarithmic derivatives of the coefficient integral and deduce the ratio inequalities. For a real coefficient index, we compare the terms in the Laplace kernel to determine the exact threshold $1/4$. From the resulting representation, we derive infinite divisibility, strict total positivity and the logarithmic difference inequalities.

The paper is organized as follows. Section~2 contains the integral representations and finite difference identities used in the proofs. In Section~3, we prove the sharp cumulant bounds and use them to establish the ratio inequalities for fixed degree. In Section~4, we determine the exact parameter range for complete monotonicity and reciprocal moment representations, and prove strict positivity of all finite Hankel minors. In Section~5, we derive the logarithmic difference inequalities and the asymptotic formulas for forward differences of the coefficients. Finally, in Section~6, we prove strict infinite log-monotonicity and infinitely divisible moment representations for the coefficient ratios after binomial normalization.

\section{Integral representations and finite differences}\label{sec:moments}

In this section, we collect the integral representations and finite-difference identities used in the proofs. For fixed $m$, we work directly with the coefficient integral. For fixed $i$, we use a factorization involving binomial coefficients and a weighted average of rising factorial ratios. The finite-difference identities introduced below allow us to relate the interpolation to the coefficient sequences.

For real $x,y$ with $x+y>-1$, define
\begin{equation}\label{eq:moment}
 M(x,y)=\int_0^1\frac{t^{y+x}}{\sqrt{1-t}\,(1+t)^{2y+3/2}}\,dt.
\end{equation}
The integral converges and is positive. With
\[
 C(y)=\frac{2^{y+1}\Gamma(2y+3/2)}{\sqrt{\pi}\,\Gamma(y+1)^2},\qquad y>-\frac{1}{2},
\]
we can write
\[
 D(x,y)=C(y)\binom yx M(x,y),\qquad x\ge 0,\quad y>x-\frac{1}{2}.
\]
The factor $C(y)$ and the binomial coefficient are positive on this domain.

The formula extends smoothly to an open neighborhood of every point of the domain, including the boundary case $x=0$, since derivatives of $M(x,y)$ introduce powers of logarithms with integrable bounds on compact subsets of $x+y>-1$. Thus, the derivatives at the endpoints of the intervals below are ordinary derivatives.

For integer $m\ge 0$, the moment formula in \cite[Lemma 3.1]{XieZhangRows} and \cite[Lemma 3.2]{XieZhangUnified} reads
\begin{equation}\label{eq:representation}
 D(i,m)=d_i(m),\qquad e_i(m):=\frac{d_i(m)}{\binom mi}=C(m)M(i,m),\qquad 0\le i\le m.
\end{equation}
Substituting $u=(1-t)/(1+t)$ in \eqref{eq:moment} yields
\begin{equation}\label{eq:moment-t}
 M(x,y)=2^{-2y-1}\int_0^1 u^{-1/2}(1-u)^{y+x}(1+u)^{y-x}\,du.
\end{equation}
At $(x,y)=(i,m)$, formula \eqref{eq:moment-t} expresses $M(i,m)$ as $2^{-2m-1}$ times the integral in \cite[eq.~(3.5)]{XieZhangUnified}.

For $x=0$, the substitution $v=u^2$ reduces \eqref{eq:moment-t} to a beta integral. Using the duplication and reflection formulas for the Gamma function \cite[Eqs.~(5.5.3) and~(5.5.5)]{NISTDLMF}, we simplify \eqref{eq:bivariate-D} to
\begin{equation}\label{eq:zero-column}
 D(0,y)=\frac{2^y\Gamma(y+3/4)}{\Gamma(3/4)\Gamma(y+1)},\qquad y>-\frac{1}{2}.
\end{equation}

To study $D(x,y)$ as a function of $y$, we use the factorization from \cite[Section~4.1, eqs.~(4.1)--(4.4)]{XieZhangUnified}. For $x>0$, define
\begin{equation}\label{eq:Q}
 Q(x,z)=\sum_{j\ge 0}p_j(x)\frac{(z)_j}{(z+1/4)_j},\qquad z\ge 0,
\end{equation}
where
\[
 p_j(x)=4^{-x-j}\frac{x}{x+j}\binom{2x+2j}{j},\qquad j\ge 0.
\]
Here $(z)_0=1$ and $(z)_j=z(z+1)\cdots(z+j-1)$ for $j\ge 1$. The weights $p_j(x)$ are positive and sum to one. For $z>0$, $Q(x,z)$ also has the integral representation
\[
 Q(x,z)=\frac{1}{B(1/4,z)}\int_0^1 v^{-3/4}(1-v)^{z-1}(1+\sqrt v)^{-2x}\,dv,
\]
where $B$ is the beta function, $B(a,b)=\Gamma(a)\Gamma(b)/\Gamma(a+b)$.

Substituting $v=u^2$ in \eqref{eq:moment-t} gives
\[
 M(x,y)=2^{-2y-2}B(1/4,y+x+1)Q(x,y+x+1).
\]
Applying the Gamma duplication and reflection formulas once more, we obtain, for $x>0$ and $y>x-1/2$,
\begin{equation}\label{eq:factorization}
 D(x,y)=D(0,y)\binom yx \frac{\binom{y+x}{x}}{\binom{y+x+1/4}{x}}Q(x,y+x+1).
\end{equation}

By \cite[Lemma~4.2]{XieZhangUnified}, for each fixed $x>0$, $-\log Q(x,z)$ is a Laplace transform in $z$. There is a positive measure $d\nu_x(t)=\eta(x,t)\,dt$ such that
\begin{gather}
 -\log Q(x,z)=\int_0^\infty e^{-zt}\,d\nu_x(t),\qquad z\ge 0, \label{eq:logQ-measure}\\
 0<t\eta(x,t)<\frac{1-e^{-t/4}}{1-e^{-t}}<1,\qquad t>0.\label{eq:logQ-density-bound}
\end{gather}
The formula at $z=0$ follows by monotone convergence, since $Q(x,z)\to p_0(x)=4^{-x}$ as $z\to0^+$. Evaluating at zero yields
\[
 \nu_x((0,\infty))=-\log Q(x,0)=x\log 4<\infty.
\]
We use this representation to estimate derivatives and finite differences of $\log D(x,y)$ with respect to $y$.

We conclude this section with the finite-difference identities that relate the interpolation to the coefficient sequences. For $h>0$, define the forward difference with respect to $x$ by
\[
 \Delta_{x,h} f(x)=f(x+h)-f(x),
\]
with all other arguments held fixed. When $h=1$, write $\Delta_x=\Delta_{x,1}$. If $f$ is $C^r$ on a neighborhood of $[x,x+rh]$, repeated application of the fundamental theorem of calculus gives
\begin{equation}\label{eq:difference-cube}
 \Delta_{x,h}^rf(x)=\int_{[0,h]^r}f^{(r)}(x+t_1+\cdots+t_r)\,dt_1\cdots dt_r,\qquad r\ge 1.
\end{equation}
In particular, if $f^{(r)}$ is strictly positive or strictly negative on $[x,x+rh]$, then $\Delta_{x,h}^rf(x)$ has the same strict sign.

For a finite positive measure $\nu$ on $[0,\infty)$, let
\[
 L(x)=\int_0^\infty e^{-xt}\,d\nu(t),\qquad x\ge 0.
\]
Taking finite differences under the integral sign yields
\begin{equation}\label{eq:laplace-difference}
 (-1)^r\Delta_{x,h}^rL(x)=\int_0^\infty e^{-xt}(1-e^{-ht})^r\,d\nu(t),\qquad x\ge 0,\quad r\ge 0.
\end{equation}
If $\nu((0,\infty))>0$, the integral is strictly positive for every $r\ge 1$.

\section{Ratios at fixed degree}\label{sec:horizontal}

In this section, we prove Theorems~\ref{thm:horizontal} and~\ref{thm:half}. We first prove strict log-concavity of $\{u_i(m)\}_{i=1}^{m-1}$ and then consider the consecutive ratios $q_i(m)$. Both proofs rely on estimates for the third and fourth derivatives of $\log M(s,m)$ with respect to $s$.

For $m\ge 4$ and $2\le i\le m-2$, the definitions of $u_i(m)$ and $D(s,m)$ yield
\begin{equation}\label{eq:fourth-difference-ratio}
 \log\frac{u_i(m)^2}{u_{i-1}(m)u_{i+1}(m)}=-\Delta_i^4\log D(i-2,m).
\end{equation}
Since $[i-2,i+2]\subseteq[0,m]$, identity \eqref{eq:difference-cube} shows that strict log-concavity of $\{u_i(m)\}_{i=1}^{m-1}$ follows from
\[
 \frac{\partial^4}{\partial s^4}\log D(s,m)<0,\qquad 0\le s\le m.
\]

Fix an integer $m\ge 1$ and write $F(s,m)=\log M(s,m)$. Taking logarithms in \eqref{eq:bivariate-D}, we obtain, for $0\le s\le m$,
\begin{equation}\label{eq:interpolation}
 \log D(s,m)=\log C(m)+\log\binom ms+F(s,m).
\end{equation}
Differentiating the Gamma formula for $\binom ms$ and writing $\psi_3=(\log\Gamma)^{(4)}$, we have
\begin{equation}\label{eq:fourth-derivative}
 \frac{\partial^4}{\partial s^4}\log D(s,m)=\partial_s^4 F(s,m)-\psi_3(s+1)-\psi_3(m-s+1).
\end{equation}
It suffices to bound $\partial_s^4 F(s,m)$ from above by the sum of the two polygamma terms. We estimate the third derivative at the same time, as it is needed for Theorem~\ref{thm:half}.

\subsection{Cumulant bounds}\label{sec:cumulants}

We express the third and fourth derivatives of $F(s,m)$ in terms of the cumulants of $-\log u$, and estimate these cumulants via integration by parts.

For $s\ge 0$, consider the probability density $w(u,s,m)/M(s,m)$ on $(0,1)$, where
\[
 w(u,s,m)=u^{m+s}(1-u)^{-1/2}(1+u)^{-2m-3/2}.
\]
Let $\E_{s,m}$ denote expectation with respect to this density. With $t=-\log u$, we define, for integrable $\phi$ and integers $j\ge 0$,
\[
 \E_{s,m}[\phi(t)]=\frac1{M(s,m)}\int_0^1 \phi(-\log u)w(u,s,m)\,du,\qquad \mu_j=\E_{s,m}[t^j].
\]
The moments $\mu_j=\mu_j(s,m)$ are finite and positive, and $\mu_0=1$. The cumulants $\kappa_j$ are defined by
\[
 \log\E_{s,m}[e^{zt}]=\sum_{j=1}^{\infty}\kappa_j\frac{z^j}{j!}.
\]
The first two cumulants are the mean and variance, and expanding through degree four yields
\[
 \kappa_3=\mu_3-3\mu_1\mu_2+2\mu_1^3,\qquad \kappa_4=\mu_4-4\mu_1\mu_3+12\mu_1^2\mu_2-6\mu_1^4-3\mu_2^2.
\]

Since $F(s,m)=\log M(s,m)$, the cumulants are directly related to its derivatives. For real $z<m+s+1$, we have
\[
 \log\E_{s,m}[e^{zt}]=\log\frac{M(s-z,m)}{M(s,m)}=F(s-z,m)-F(s,m).
\]
Taking the third and fourth derivatives with respect to $z$ at zero gives
\begin{equation}\label{eq:cumulant}
 \partial_s^3 F(s,m)=-\kappa_3,\qquad \partial_s^4 F(s,m)=\kappa_4.
\end{equation}

Under $u=e^{-t}$, the unnormalized measure becomes $\omega(t,s,m)\,dt$, where $\omega(t,s,m)=e^{-t}w(e^{-t},s,m)$. Its logarithmic derivative is
\[
 \partial_t\log\omega(t,s,m)=-\frac1{2t}-\rho(t,s,m),
\]
where
\[
 \rho(t,s,m)=s+\frac{e^t+1}{4(e^t-1)}-\frac1{2t}+\left(m+\frac34\right)\frac{e^t-1}{e^t+1}.
\]
Since $(e^t-1)^2>t^2e^t$ for $t>0$, we have $\partial_t\rho(t,s,m)>0$. The function $\rho(t,s,m)$ tends to $s$ as $t\to0^+$ and to $s+m+1$ as $t\to\infty$. Hence
\[
 s<\rho(t,s,m)<s+m+1,\qquad t>0.
\]

For any probability measure on a real interval, suppose that $g$ and $h$ have opposite monotonicity, $h\ge0$, and the expectations below are finite with $\E[h]>0$. Then
\begin{equation}\label{eq:monotone-weight}
 \frac{\E[gh]}{\E[h]}\le\E[g].
\end{equation}
If both functions are nondecreasing, the inequality is reversed. In either case, integrate $(g(u)-g(v))(h(u)-h(v))$ against the product measure and use its sign.

\begin{lemma}\label{lem:cumulant-bound}
For every integer $m\ge 1$ and every $s\ge 0$,
\begin{equation}\label{eq:cumulant-comparisons}
 0<\kappa_3\le\frac8{15}\mu_3,\qquad \kappa_4\le\frac{16}{35}\mu_4.
\end{equation}
\end{lemma}

\begin{proof}
Write $\mu=\mu_1$. For each integer $j\ge0$, the logarithmic
derivative of $\omega(t,s,m)$ with respect to $t$ implies
\[
 \partial_t\bigl[t^{j+1}\omega(t,s,m)\bigr]
 =\bigl[(j+\tfrac12)t^j-t^{j+1}\rho(t,s,m)\bigr]\omega(t,s,m).
\]
The boundary term $t^{j+1}\omega(t,s,m)$ is
$O(t^{j+1/2})$ at zero and decays exponentially at infinity.
The bound on $\rho(t,s,m)$ also ensures absolute integrability of
the terms on the right. Integration by parts thus gives
\begin{equation}\label{eq:row-ibp}
 \E_{s,m}\bigl[t^{j+1}\rho(t,s,m)\bigr]
       =\left(j+\frac12\right)\mu_j,\qquad j\ge0.
\end{equation}
To prove that $\kappa_3>0$, we express $4\kappa_3$ as the
expectation of a nonnegative function. For any real $c$, we
expand $t(t-c)^2$ and apply \eqref{eq:row-ibp} with $j=1,2,3$
to obtain
\[
 \E_{s,m}\bigl[t(t-c)^2\bigl(t\rho(t,s,m)+1/2\bigr)\bigr]
 =4\mu_3-6c\mu_2+2c^2\mu.
\]
For $c=2\mu$, the expression on the right is $4\kappa_3$, and hence
\[
 4\kappa_3=\E_{s,m}\bigl[
 t(t-2\mu)^2\bigl(t\rho(t,s,m)+1/2\bigr)\bigr]>0.
\]
Here $\rho(t,s,m)>s\ge0$, and the integrand vanishes only at
$t=2\mu$ in $(0,\infty)$. The expectation is strictly positive
because the density is positive throughout this interval.

For the upper bounds, fix an integer $r\ge1$ and use the probability
measure proportional to $t^r\omega(t,s,m)\,dt$. Since $t$ and
$\rho(t,s,m)$ are both increasing, the reversed inequality in
\eqref{eq:monotone-weight} gives
\[
 \mu_{r+1}\E_{s,m}[t^r\rho(t,s,m)]
 \le \mu_r\E_{s,m}[t^{r+1}\rho(t,s,m)].
\]
By \eqref{eq:row-ibp}, the two expectations are
$(r-\tfrac12)\mu_{r-1}$ and $(r+\tfrac12)\mu_r$. In particular,
\begin{equation}\label{eq:normalized-moment-bounds}
 \mu_2\le3\mu^2,\qquad
 \mu_3\le\frac{5\mu_2^2}{3\mu},\qquad
 \mu_4\le\frac{7\mu_3^2}{5\mu_2}.
\end{equation}

Define the normalized moments
\[
 U=\frac{\mu_2}{\mu^2},\qquad V=\frac{\mu_3}{\mu^3}.
\]
Together with $\mu_2\ge\mu^2$ and the Cauchy--Schwarz inequality
$\mu_2^2\le\mu\mu_3$, the preceding bounds imply that
\begin{equation}\label{eq:moment-region}
 1\le U\le3,\qquad U^2\le V\le\frac53U^2.
\end{equation}
Consequently,
\[
 \frac{\kappa_3-\frac8{15}\mu_3}{\mu^3}
 =\frac7{15}V-3U+2
 \le\frac{(U-3)(7U-6)}9\le0.
\]

For the fourth cumulant, substitution of the last bound in
\eqref{eq:normalized-moment-bounds} into its defining formula yields
\[
 \frac{\kappa_4-\frac{16}{35}\mu_4}{\mu^4}
 \le T(U,V),\qquad
 T(U,V)=\frac{19V^2}{25U}-4V+12U-6-3U^2.
\]
For fixed $U\in[1,3]$, this quadratic is convex in $V$, so we need
only check the two endpoints of the interval in
\eqref{eq:moment-region}. At the upper endpoint,
\[
 T\left(U,\frac53U^2\right)
 =\frac{(U-3)(19U^2-30U+18)}9\le0.
\]
The quadratic factor is positive, since
\[
 19U^2-30U+18=19(U-1)^2+8(U-1)+7>0.
\]
At the lower endpoint, write $U=1+z$ with $0\le z\le2$. Then
\begin{align*}
 25T(U,U^2)
 &=19z^3-118z^2+7z-6\\
 &\le-80z^2+7z-6
 =-(4z-1)^2-64z^2-z-5<0.
\end{align*}
\end{proof}

The constants $8/15$ and $16/35$ are optimal over $s\ge0$ for each fixed integer $m\ge1$. For $s>0$, \eqref{eq:row-ibp} and the bounds on $\rho$ imply
\[
 s\mu_{j+1}<\left(j+\frac12\right)\mu_j
 <(s+m+1)\mu_{j+1},\qquad j\ge0.
\]
Starting from $\mu_0=1$, we obtain
\begin{equation}\label{eq:moment-sandwich}
 \frac{(1/2)_j}{(s+m+1)^j}<\mu_j(s,m)<\frac{(1/2)_j}{s^j},\qquad j\ge1,
\end{equation}
where $(1/2)_j=\Gamma(j+1/2)/\Gamma(1/2)$. For fixed $m$, $s^j\mu_j(s,m)\to(1/2)_j$ as $s\to\infty$. Substituting these limits into the cumulant formulas gives $s^3\kappa_3\to1$ and $s^4\kappa_4\to3$, so the ratios converge to $8/15$ and $16/35$.

\subsection{The fourth derivative}\label{sec:fourth-derivative}

We now prove that $\partial_s^4\log D(s,m)<0$ throughout $0\le s\le m$. The cumulant bound and \eqref{eq:moment-sandwich} control $\partial_s^4 F(s,m)$ away from zero. Near zero, we return to the beta integral.

\begin{proposition}\label{prop:moment-bounds}
For every integer $m\ge2$ and every $s\ge0$, let $a=m+s$. Then
\begin{equation}\label{eq:beta-moment-bounds}
 \mu_2(s,m)<\frac1a,\qquad
 \mu_4(s,m)<\frac{5(a-1/4)}{a(a-1)(a+5/4)}.
\end{equation}
In particular,
\[
 \mu_4(s,m)<\frac5{m(m-1)}.
\]
If $a\ge5/2$, we also have
\begin{equation}\label{eq:joint-fourth-moment}
 \mu_4(s,m)<\frac5{a^2}.
\end{equation}
For $m\ge1$ and $s>0$,
\[
 \mu_3(s,m)<\frac{15}{8s^3},\qquad \mu_4(s,m)<\frac{105}{16s^4}.
\]
\end{proposition}

\begin{proof}
We repeat the two substitutions used to obtain \eqref{eq:zero-column},
writing $v=(1-u)/(1+u)$ and $z=v^2$. Under the probability measure
defining $\E_{s,m}$, the density of $z$ is proportional to
\[
 z^{-3/4}(1-z)^a(1+\sqrt z)^{-2s},\qquad 0<z<1,
\]
and $t=-\log u=2\operatorname{arctanh}\sqrt z$.
The factor $(1+\sqrt z)^{-2s}$ is nonincreasing. By
\eqref{eq:monotone-weight}, the expectation of an increasing
function of $z$ is at most its expectation under the beta
distribution with parameters $1/4$ and $a+1$, whose density is
\[
 \frac{z^{-3/4}(1-z)^a}{B(1/4,a+1)},\qquad 0<z<1.
\]

We next bound the second and fourth powers of
$2\operatorname{arctanh}\sqrt z$. For $0<v<1$, concavity of
the cube root implies
\[
 \frac{d}{dv}\left(\frac{v}{(1-v^2)^{1/3}}
                   -\operatorname{arctanh}v\right)
 =\frac{1-v^2/3-(1-v^2)^{1/3}}{(1-v^2)^{4/3}}>0.
\]
The expression in parentheses vanishes at zero, so
$\operatorname{arctanh}v<v(1-v^2)^{-1/3}$.
Using also $(1-z)^{2/3}\le1-2z/3$, we obtain
\[
 (2\operatorname{arctanh}\sqrt z)^2<\frac{4z}{1-z},\qquad
 (2\operatorname{arctanh}\sqrt z)^4
 <\frac{16z^2(1-2z/3)}{(1-z)^2}.
\]
Both powers on the left are increasing. We first compare their
expectations with those under the beta distribution, and then
integrate the displayed bounds. Since $a\ge2$, all the beta
integrals converge, and
\begin{align*}
 \mu_2(s,m)&<\frac{4B(5/4,a)}{B(1/4,a+1)}=\frac1a,\\
 \mu_4(s,m)&<\frac{16\bigl[B(9/4,a-1)-\tfrac23B(13/4,a-1)\bigr]}
                     {B(1/4,a+1)}
 =\frac{5(a-1/4)}{a(a-1)(a+5/4)}.
\end{align*}
The bound $5/[m(m-1)]$ follows from $a\ge m$ and
$(a-1/4)/(a+5/4)<1$. For $a\ge5/2$, the identity
\[
 (a-1)(a+5/4)-a(a-1/4)=\frac a2-\frac54\ge0
\]
gives \eqref{eq:joint-fourth-moment}. The bounds in terms of $s$
follow directly from
\eqref{eq:moment-sandwich} with $j=3,4$.
\end{proof}

We now compare $\partial_s^4 F(s,m)$ with the polygamma terms in
\eqref{eq:fourth-derivative}.
For $u>0$, the function $(u+t)^{-4}$ is strictly convex in $t$
on $[0,\infty)$.
Summing the trapezoidal bounds over the unit intervals gives
\[
 \sum_{j\ge0}(u+j)^{-4}
 >\int_0^\infty(u+t)^{-4}\,dt+\frac1{2u^4}.
\]
Evaluating the integral, we obtain
\begin{equation}\label{eq:polygamma-bound}
 \psi_3(u)=6\sum_{j=0}^\infty\frac1{(u+j)^4}
 >\frac2{u^3}+\frac3{u^4},\qquad u>0.
\end{equation}

\begin{proposition}\label{prop:fourth-derivative}
For every integer $m\ge 4$,
\begin{equation}\label{eq:one-sided-fourth}
 \partial_s^4 F(s,m)<\psi_3(s+1),\qquad s\ge 0.
\end{equation}
Consequently, $\partial_s^4\log D(s,m)<0$ for $0\le s\le m$.
\end{proposition}

\begin{proof}
By Lemma~\ref{lem:cumulant-bound} and Proposition~\ref{prop:moment-bounds},
\begin{align*}
 \partial_s^4 F(s,m)&<\frac{16}{7(m+s)^2}\le\frac{16}{7(s+4)^2},
 &&s\ge0,\\
 \partial_s^4 F(s,m)&<\frac3{s^4},&&s>0.
\end{align*}
We compare these bounds with
\[
 R(s)=\frac{2s+5}{(s+1)^4}
     =\frac2{(s+1)^3}+\frac3{(s+1)^4}.
\]
Direct differentiation gives
\begin{align*}
 \frac{d}{ds}\bigl[(s+4)^2R(s)\bigr]
 &=-\frac{2(s+4)(s^2+14s+31)}{(s+1)^5}<0,\\
 \frac{d}{ds}\bigl[s^4R(s)\bigr]
 &=\frac{2s^3(s^2+5s+10)}{(s+1)^5}>0\qquad(s>0).
\end{align*}
At $c=11/4$, the two products satisfy
\[
 (c+4)^2R(c)=\frac{1512}{625}>\frac{16}7,
 \qquad c^4R(c)=\frac{102487}{33750}>3.
\]
It follows that
\begin{align*}
 \partial_s^4 F(s,m)&<\frac{16}{7(s+4)^2}<R(s),&&0\le s\le c,\\
 \partial_s^4 F(s,m)&<\frac3{s^4}<R(s),&&s\ge c.
\end{align*}
Since $R(s)<\psi_3(s+1)$ by \eqref{eq:polygamma-bound},
the estimate \eqref{eq:one-sided-fourth} holds for all $s\ge0$.
On $[0,m]$, the remaining term $\psi_3(m-s+1)$ in
\eqref{eq:fourth-derivative} is positive, so the fourth derivative
of $\log D(s,m)$ is strictly negative.
\end{proof}

\begin{proof}[\normalfont\bfseries Proof of Theorem~\ref{thm:horizontal}]
For $2\le i\le m-2$, Proposition~\ref{prop:fourth-derivative} and \eqref{eq:difference-cube} show that the right-hand side of \eqref{eq:fourth-difference-ratio} is positive.
\end{proof}

\begin{corollary}\label{cor:arbitrary-step}
Let $m\ge4$ be an integer and $0<h<m/2$. Then
\[
 U(s,m,h)=\frac{D(s-h,m)D(s+h,m)}{D(s,m)^2},
 \qquad h\le s\le m-h,
\]
is strictly log-concave as a function of $s$. Its values along any arithmetic progression
of positive step contained in $[h,m-h]$ form a strictly log-concave
sequence at the interior indices.
\end{corollary}

\begin{proof}
For $h\le s\le m-h$, the same fourth derivative estimate gives
\[
 \partial_s^2\log U(s,m,h)=\Delta_{s,h}^2\partial_s^2\log D(s-h,m)<0
\]
by \eqref{eq:difference-cube}. Thus $\log U(s,m,h)$ is strictly
concave in $s$ on $[h,m-h]$. For any three consecutive points of an
arithmetic progression in this interval, the strict midpoint
inequality gives the required strict log-concavity.
\end{proof}

\subsection{Ratios of consecutive coefficients}\label{sec:half}

To locate the change from strict log-convexity to strict log-concavity of $\{q_i(m)\}_{i=1}^m$, we use the factorization
\[
 q_i(m)=\frac{m-i+1}{i}\frac{M(i,m)}{M(i-1,m)}.
\]
The second difference of $\log(M(i,m)/M(i-1,m))$ centered at $i$ is $\Delta_i^3F(i-2,m)$.

\begin{corollary}\label{cor:third-derivative}
For $m\ge1$, $\partial_s^3F(s,m)<0$ for $s\ge0$ and
\begin{equation}\label{eq:third-bound}
 -s^{-3}<\partial_s^3F(s,m)<0,\qquad s>0.
\end{equation}
For $m\ge4$, the stronger bound
\begin{equation}\label{eq:shifted-third-bound}
 -(s+1/2)^{-3}<\partial_s^3F(s,m)<0,\qquad 0\le s\le m-1,
\end{equation}
also holds.
\end{corollary}

\begin{proof}
Recall that $\partial_s^3 F(s,m)=-\kappa_3$.
Lemma~\ref{lem:cumulant-bound} gives $\kappa_3>0$ for all $s\ge0$.
For $s>0$, that lemma and
Proposition~\ref{prop:moment-bounds} also give
\[
 0<\kappa_3\le\frac8{15}\mu_3<\frac1{s^3}.
\]

To prove \eqref{eq:shifted-third-bound}, we retain part of the
dependence on $m$ in $\rho(t,s,m)$. For $t>0$, we have
\[
 \frac{e^t+1}{e^t-1}>\frac2t,
 \qquad
 \frac{e^t-1}{e^t+1}>\frac{t}{t+2}.
\]
The first inequality follows from $\tanh(t/2)<t/2$, and the second
from $e^t>1+t$. Consequently,
\[
 \rho(t,s,m)>s+\left(m+\frac34\right)\frac{t}{t+2}.
\]
For $j=0,1,2$, let $\alpha_j=j+1/2$ and
$r_j=\mu_{j+1}/\mu_j$. Under the probability measure proportional
to $t^j\omega(t,s,m)\,dt$, the mean of $t$ is $r_j$.
The function $t^2/(t+2)$ is strictly convex, since its second
derivative is $8/(t+2)^3$. Applying Jensen's inequality and
\eqref{eq:row-ibp}, we obtain
\[
 \alpha_j>s r_j+\left(m+\frac34\right)\frac{r_j^2}{r_j+2}.
\]

We compare $r_j$ with $\alpha_j/(s+c_j)$, where
$(c_0,c_1,c_2)=(1/4,1/2,1)$.
For $m\ge4$ and $s\le m-1$,
\[
 \alpha_j\left(m+\frac34\right)
       -c_j(\alpha_j+2s+2c_j)
 \ge
 \begin{cases}
 5/8,&j=0,\\
 m/2+7/8,&j=1,\\
 m/2-5/8,&j=2,
 \end{cases}
\]
so all three differences are positive. Substituting
$r=\alpha_j/(s+c_j)$ into
$sr+(m+3/4)r^2/(r+2)$ therefore gives a value greater than
$\alpha_j$. This expression is strictly increasing for $r>0$,
and hence $r_j<\alpha_j/(s+c_j)$.
Multiplication of the three inequalities yields
\[
 \mu_3=r_0r_1r_2
 <\frac{15}{8(s+1/4)(s+1/2)(s+1)}
 \le\frac{15}{8(s+1/2)^3}.
\]
Here the last step uses
$(s+1/4)(s+1)-(s+1/2)^2=s/4\ge0$.
Combining this estimate with $\kappa_3\le8\mu_3/15$ proves
\eqref{eq:shifted-third-bound}.
\end{proof}

The consecutive ratios of $e_i(m)=d_i(m)/\binom mi$ are
\[
 v_i(m)=\frac{M(i,m)}{M(i-1,m)}.
\]
By Corollary~\ref{cor:third-derivative},
\begin{equation}\label{eq:normalized-curvature}
 \log\frac{v_{i-1}(m)v_{i+1}(m)}{v_i(m)^2}=\Delta_i^3F(i-2,m)<0,\qquad 2\le i\le m-1.
\end{equation}
Since $q_i(m)=(m-i+1)v_i(m)/i$, we have
\begin{equation}\label{eq:curvature-factor}
 \frac{q_{i-1}(m)q_{i+1}(m)}{q_i(m)^2}=\beta_i(m)\frac{v_{i-1}(m)v_{i+1}(m)}{v_i(m)^2},
\end{equation}
where
\begin{equation}\label{eq:B}
 \beta_i(m)=\frac{i^2}{i^2-1}\left(1-\frac1{(m-i+1)^2}\right),\qquad \beta_i(m)-1=\frac{(m+1)(m+1-2i)}{(i^2-1)(m-i+1)^2}.
\end{equation}

\begin{lemma}\label{lem:midpoint}
For every integer $k\ge3$,
\[
 \Delta_k^3\log(k-1)>\frac12\Delta_k^3\log(k-3/2).
\]
\end{lemma}

\begin{proof}
Applying \eqref{eq:difference-cube} to $\log x$ and evaluating
two of the three integrals, we obtain
\[
 \Delta_x^3\log x=2\int_0^1
 \frac{dt}{(x+t)(x+t+1)(x+t+2)},\qquad x>0.
\]
To compare the integrands at $x=k-1$ and $x=k-3/2$, let
$w=k-3/2+t$. Since $w\ge3/2$,
\[
 \frac{(w+1/2)(w+3/2)(w+5/2)}{w(w+1)(w+2)}
 \le\frac43\cdot\frac65\cdot\frac87
 =\frac{64}{35}<2.
\]
Thus the first integrand is greater than half the second
throughout $0\le t\le1$. Integrating this strict inequality,
we obtain the stated comparison.
\end{proof}

\begin{proof}[\normalfont\bfseries Proof of Theorem~\ref{thm:half}]
Set $k=\lfloor m/2\rfloor$ and, for $2\le i\le m-1$, write
\[
 K_i=\log\frac{q_{i-1}(m)q_{i+1}(m)}{q_i(m)^2}
     =\log\beta_i(m)+\Delta_i^3 F(i-2,m).
\]
If $k+1\le i\le m-1$, then $2i\ge m+1$ and $\beta_i(m)\le1$.
By \eqref{eq:normalized-curvature}, $K_i<0$, proving
\eqref{eq:right-half}. The inequality remains strict at
$i=(m+1)/2$ for odd $m$, where $\beta_i(m)=1$.
When $m=3$, the left-hand range is empty, so the proof is complete.

Suppose now that $m\ge4$. The identity $u_i(m)=q_{i+1}(m)/q_i(m)$
and Theorem~\ref{thm:horizontal} imply
\begin{equation}\label{eq:curvature-decreasing}
 K_{i+1}-K_i
 =\log\frac{u_{i-1}(m)u_{i+1}(m)}{u_i(m)^2}<0,
 \qquad 2\le i\le m-2.
\end{equation}
For fixed $k\ge2$, the factor $\beta_k(m)$ increases with $m\ge2k$,
and its value at $m=2k$ satisfies
\[
 \log\beta_k(2k)
 =\log\frac{k^3(k+2)}{(k-1)(k+1)^3}
 =\Delta_k^3\log(k-1).
\]

If $k\ge3$, then $[k-2,k+1]\subseteq[0,m-1]$.
The bound \eqref{eq:shifted-third-bound} shows that
$F(s,m)+\tfrac12\log(s+1/2)$ has positive third derivative
on this interval. By \eqref{eq:difference-cube},
\[
 K_k>\Delta_k^3\log(k-1)
       -\tfrac12\Delta_k^3\log(k-3/2)>0,
\]
where the last inequality is Lemma~\ref{lem:midpoint}.

It remains to consider $k=2$, which occurs when $m=4$ or $5$.
For these degrees and all $s\ge0$, Cauchy--Schwarz and the two
moment bounds in Proposition~\ref{prop:moment-bounds} give
\[
 \kappa_3\le\frac8{15}\mu_3
 \le\frac8{15}\sqrt{\mu_2\mu_4}
 <\frac{8\sqrt5}{15(m+s)^{3/2}}\le\frac{\sqrt5}{15}.
\]
Since $\partial_s^3 F(s,m)=-\kappa_3$, integration over $[0,1]^3$ and
$\beta_2(m)\ge\beta_2(4)=32/27$ imply
\[
 K_2>\log\frac{32}{27}-\frac{\sqrt5}{15}
 >\frac5{32}-\frac3{20}=\frac1{160}>0.
\]
Here we used $\log(1+x)>x/(1+x)$ for $x>0$ and $\sqrt5<9/4$.
In all cases, \eqref{eq:curvature-decreasing} now gives
$K_i\ge K_k>0$ for $2\le i\le k$, proving \eqref{eq:left-half}.
\end{proof}

\section{Reciprocal moments and strict total positivity}\label{sec:columns}

In this section, we prove Theorem~\ref{thm:threshold-moments} and establish strict total positivity for the reciprocal coefficients. For $\lambda\ge0$ and $z>-1/2$, define
\[
 \Psi(\lambda,z)=\log\frac{D(\lambda,\lambda+z)}{2^zD(\lambda,\lambda)}.
\]
This normalization satisfies $\Psi(\lambda,0)=0$, and
\[
 a_n(\lambda)=e^{-\Psi(\lambda,n)},\qquad \lambda>0,\quad n\ge0.
\]
The derivatives of $\Psi$ with respect to $z$ satisfy
\begin{align*}
 \partial_z\Psi(\lambda,z)
 &=\partial_z\log D(\lambda,\lambda+z)-\log2,\\
 \partial_z^r\Psi(\lambda,z)
 &=\partial_z^r\log D(\lambda,\lambda+z),\qquad r\ge2.
\end{align*}
Thus the complete monotonicity assertion concerns $-\partial_z^2\Psi(\lambda,z)$ on $[0,\infty)$. We first record the growth of $D(\lambda,y)$, which excludes the moment property when $0<\lambda<1/4$.

Fix $\lambda>0$. The summand indexed by $j$ in \eqref{eq:Q} is bounded above by $p_j(\lambda)$ and tends to $p_j(\lambda)$ as $z\to\infty$. Since $\sum_{j\ge0}p_j(\lambda)=1$, dominated convergence gives $Q(\lambda,z)\to1$. We express the binomial coefficients in \eqref{eq:factorization} in terms of Gamma functions. Applying the asymptotic formula for ratios of Gamma functions \cite[Eq.~(5.11.12)]{NISTDLMF}, we obtain
\begin{equation}\label{eq:column-growth}
 D(\lambda,y)\sim\frac{2^y y^{\lambda-1/4}}{\Gamma(\lambda+1)\Gamma(3/4)},\qquad y\to\infty.
\end{equation}
In \eqref{eq:reciprocal-column}, the factor $2^n$ removes this geometric growth, and the factor $D(\lambda,\lambda)$ normalizes the zeroth term to one. For $0<\lambda<1/4$, the asymptotic estimate implies $a_n(\lambda)\to\infty$, so these numbers cannot be the moments of a probability measure on $[0,1]$.

\subsection{The derivatives of \texorpdfstring{$\Psi$}{Psi}}\label{subsec:column-curvature}

We derive a Laplace representation for $\partial_z\Psi$ and use it to determine the signs of the higher derivatives. To exclude $0<\lambda<1/4$, we also need the asymptotic behavior of $\partial_z^2\Psi(\lambda,z)$. For background on complete monotonicity of ratios of Gamma functions, see Bustoz and Ismail \cite{BustozIsmail}.

We start with the kernel for normalized coefficients in \cite[Lemma~4.3]{XieZhangUnified} and compare it with the kernel at coefficient index zero. For $t>0$, write $u=e^{-t/4}$ and $S=1+u+u^2+u^3$, and define
\begin{equation}\label{eq:coefficient-kernel}
 k(\lambda,t)=
 \begin{cases}
 \displaystyle\frac{u^3-u^4+u^{4\lambda+4}}S -tu^{4\lambda+4}\eta(\lambda,t),&\lambda>0,\\[6pt]
 \displaystyle\frac{u^3}S,&\lambda=0.
 \end{cases}
\end{equation}

\begin{lemma}\label{lem:coefficient-kernel}
For every real $\lambda\ge0$ and $y>\lambda-1/2$,
\begin{equation}\label{eq:normalized-log-derivative}
 -\frac{\partial}{\partial y}\log\frac{D(\lambda,y)}{2^y\binom y\lambda} =\int_0^\infty e^{-yt}k(\lambda,t)\,dt.
\end{equation}
For $\lambda>0$ and $t>0$,
\begin{equation}\label{eq:coefficient-kernel-bounds}
 0<\frac{u^3-u^4}S<k(\lambda,t)<k(0,t)=\frac{e^{t/4}-1}{e^t-1}.
\end{equation}
For each fixed $\lambda\ge0$, the function $k(\lambda,t)$ is bounded as $t\to0^+$ and is $O(e^{-3t/4})$ as $t\to\infty$.
\end{lemma}

\begin{proof}
For $\lambda>0$, divide \eqref{eq:factorization} by $2^y\binom y\lambda$ and take the negative logarithmic derivative. Using \eqref{eq:zero-column} and the Gamma formula for the binomial coefficients, we apply the digamma integral representation \cite[Eq.~(5.9.16)]{NISTDLMF} to the resulting differences. The contribution to the kernel from these factors is
\[
 \frac{e^{-3t/4}+e^{-5t/4}-2e^{-t}+e^{-(\lambda+1)t}-e^{-(\lambda+5/4)t}}{1-e^{-t}} =\frac{u^3-u^4+u^{4\lambda+4}}S.
\]
The factor $Q(\lambda,y+\lambda+1)$ contributes $-te^{-(\lambda+1)t}\eta(\lambda,t)$ by \eqref{eq:logQ-measure}. The resulting kernel is $k(\lambda,t)$, so \eqref{eq:normalized-log-derivative} holds.

The density bound \eqref{eq:logQ-density-bound} reads $0<t\eta(\lambda,t)<1/S$. It follows that
\[
 \frac{u^3-u^4}S<k(\lambda,t)<\frac{u^3-u^4+u^{4\lambda+4}}S<\frac{u^3}S,
\]
since $0<u<1$ and $\lambda>0$. For $\lambda=0$, the integral identity follows directly from \eqref{eq:zero-column}. The bound $k(\lambda,t)\le k(0,t)=u^3/S$ proves the stated estimates at zero and infinity. These estimates also justify differentiation under the integral sign to every order on $y>\lambda-1/2$.
\end{proof}

To obtain the first derivative of $\Psi$, we also need the logarithmic derivative of the binomial factor. For $\lambda\ge0$ and $y>\lambda-1/2$, it has the integral representation
\[
 \frac{\partial}{\partial y}\log\binom y\lambda =\int_0^\infty e^{-yt}\frac{e^{\lambda t}-1}{e^t-1}\,dt.
\]
Combining this identity with \eqref{eq:normalized-log-derivative} and substituting $y=\lambda+z$, we obtain, for $z>-1/2$,
\begin{equation}\label{eq:column-first-derivative}
 \partial_z\Psi(\lambda,z) =\int_0^\infty e^{-(\lambda+z)t}h(\lambda,t)\,dt, \qquad h(\lambda,t)=\frac{e^{\lambda t}-1}{e^t-1}-k(\lambda,t).
\end{equation}
At $\lambda=1/4$, the first term in $h(\lambda,t)$ is exactly $k(0,t)$. Since this term increases with $\lambda$, the strict upper bound for $k(\lambda,t)$ implies
\begin{equation}\label{eq:column-kernel-signs}
 h(\lambda,t)>0\quad(\lambda\ge1/4),\qquad h(0,t)=-k(0,t)<0,\qquad t>0.
\end{equation}
In particular, positivity holds at the critical parameter as well.

We next determine the asymptotic behavior of $\partial_z^2\Psi$ for all $\lambda>0$. Define
\begin{equation}\label{eq:column-kernel}
 b(\lambda,t)=e^{-t}-e^{-(\lambda+3/4)t}+(1-e^{-t/4})\bigl(e^{-(\lambda+1)t}-e^{-(2\lambda+1)t}\bigr).
\end{equation}
\begin{proposition}\label{prop:column-curvature}
For every $\lambda>0$ and $z>-1/2$,
\begin{equation}\label{eq:column-curvature-laplace}
 \begin{split}
 -\partial_z^2\Psi(\lambda,z)={}&\int_0^\infty e^{-zt}
 \frac{tb(\lambda,t)}{1-e^{-t}}\,dt\\
 &+\int_0^\infty t^2e^{-(z+2\lambda+1)t}\,d\nu_\lambda(t).
 \end{split}
\end{equation}
Moreover,
\begin{equation}\label{eq:critical-curvature-limit}
 \lim_{z\to\infty}z^2\partial_z^2\Psi(\lambda,z)=\frac14-\lambda.
\end{equation}
Accordingly, the function $D(\lambda,y)$ is eventually strictly log-convex in $y$ when $0<\lambda<1/4$.
\end{proposition}

\begin{proof}
The definitions of $h(\lambda,t)$ and $b(\lambda,t)$ give
\[
 e^{-\lambda t}h(\lambda,t)=\frac{b(\lambda,t)}{1-e^{-t}}+te^{-(2\lambda+1)t}\eta(\lambda,t).
\]
Differentiating \eqref{eq:column-first-derivative} with respect to $z$ and using this identity, we obtain \eqref{eq:column-curvature-laplace}. Since $b(\lambda,t)=O(t)$ as $t\to0^+$, both integrands decay exponentially at infinity, and $\nu_\lambda$ is finite, we may differentiate under both integrals to every order, uniformly on compact subintervals of $(-1/2,\infty)$.

To evaluate the limit, let $\Phi(\lambda,t)=b(\lambda,t)/(1-e^{-t})$. This function is bounded on $(0,\infty)$ and satisfies
\[
 \lim_{t\to0^+}\Phi(\lambda,t)=\lambda-\frac14, \qquad \lim_{t\to\infty}\Phi(\lambda,t)=0.
\]
After the substitution $u=zt$, dominated convergence gives
\[
 z^2\int_0^\infty e^{-zt}t\Phi(\lambda,t)\,dt =\int_0^\infty ue^{-u}\Phi(\lambda,u/z)\,du \longrightarrow\lambda-\frac14.
\]
For the other integral, the factor $(zt)^2e^{-zt}$ tends to zero for every $t>0$ and is bounded by $4e^{-2}$. Since $\nu_\lambda$ is finite, $z^2\int_0^\infty t^2e^{-(z+2\lambda+1)t}\,d\nu_\lambda(t)\to0$. Returning to \eqref{eq:column-curvature-laplace}, we obtain \eqref{eq:critical-curvature-limit}. When $0<\lambda<1/4$, the limit is positive, so $\partial_z^2\Psi(\lambda,z)>0$ for all sufficiently large $z$.
\end{proof}

\subsection{Hausdorff moments and Hankel minors}\label{sec:moment-extensions}

For $\lambda\ge1/4$, we use Bernstein functions to construct probability measures on $[0,1]$ representing every positive power of $\{a_n(\lambda)\}_{n\ge0}$. Strict positivity of all Hankel minors will follow once we show that these measures have infinitely many positive support points.

For a Hausdorff moment sequence $m_n=\int_0^1t^n\,d\sigma(t)$, suppose that the support of $\sigma$ contains infinitely many points in $(0,1]$. Then every finite minor of $[m_{i+j}]_{i,j\ge0}$ is positive \cite[Theorem~2.8]{FallatJohnsonSokal}. For increasing nonnegative integer indices $r_1<\cdots<r_k$ and $s_1<\cdots<s_k$, determinant expansion and Fubini's theorem give
\begin{equation}\label{eq:moment-minors}
 \det[m_{r_i+s_j}]_{i,j=1}^k =\frac1{k!}\int_{[0,1]^k} \det[t_j^{r_i}]_{i,j=1}^k\det[t_j^{s_i}]_{i,j=1}^k \,d\sigma(t_1)\cdots d\sigma(t_k).
\end{equation}
Each generalized Vandermonde determinant is the Vandermonde product times a Schur polynomial, and is therefore positive when $0<t_1<\cdots<t_k$. Their product is symmetric in the variables and nonnegative on the whole cube by continuity. Choose $k$ distinct support points in $(0,1]$ and pairwise disjoint relative neighborhoods of these points, ordered from left to right. The integrand is strictly positive on their product, which has positive $\sigma^k$-measure, so the determinant is strictly positive.

A nonnegative smooth function $f$ on $(0,\infty)$ is a Bernstein function if $f'$ is completely monotone. For a Bernstein function $f$ with $f(0^+)=0$, $e^{-\tau f}$ is the Laplace transform of a probability measure on $[0,\infty)$ for every $\tau>0$ \cite[Theorems~1.4 and~3.7]{SchillingSongVondracek}; the zero function corresponds to the point mass at zero.

\begin{lemma}\label{lem:bernstein-moments}
Let $f$ be a Bernstein function with $f(0^+)=0$, and suppose that $f$ is not of the form $f(x)=ax$ for a constant $a\ge0$. For every $\tau>0$, there is a probability measure $\sigma_\tau$ on $[0,1]$ with infinitely many positive support points such that
\[
 e^{-\tau f(n)}=\int_0^1u^n\,d\sigma_\tau(u),\qquad n\ge0,
\]
where $f(0)=0$. The Hankel matrix $[e^{-\tau f(i+j)}]_{i,j\ge0}$ is strictly totally positive.
\end{lemma}

\begin{proof}
By the Bernstein representation, there is a probability measure $\mu_\tau$ on $[0,\infty)$ such that $e^{-\tau f(x)}=\int_0^\infty e^{-xt}\,d\mu_\tau(t)$. For every positive integer $N$, uniqueness of the Laplace transform implies $\mu_\tau=\mu_{\tau/N}^{*N}$.

We claim that $\mu_\tau$ has unbounded support. Suppose instead that its support is contained in $[0,B]$. The convolution identity forces $\mu_{\tau/N}$ to be supported in $[0,B/N]$. Let $X_1,\ldots,X_N$ be independent with distribution $\mu_{\tau/N}$. Their sum has distribution $\mu_\tau$, so
\[
 \operatorname{Var}(X_1+\cdots+X_N)=N\operatorname{Var}(X_1)\le\frac{B^2}{N}.
\]
Letting $N\to\infty$ shows that $\mu_\tau$ has variance zero, hence is a point mass $\delta_b$. Its Laplace transform implies $f(x)=bx/\tau$, a contradiction.

Push $\mu_\tau$ forward under $v=e^{-t}$ to obtain $\sigma_\tau$. Its moments are $e^{-\tau f(n)}$. Since $\mu_\tau$ has unbounded support, $\sigma_\tau$ has infinitely many positive support points. Strict total positivity now follows from \eqref{eq:moment-minors}.
\end{proof}

\begin{proof}[\normalfont\bfseries Proof of Theorem~\ref{thm:threshold-moments}]
Suppose first that $0<\lambda<1/4$. By \eqref{eq:critical-curvature-limit}, $\partial_z^2\Psi(\lambda,z)>0$ for all sufficiently large $z$, so $-\partial_z^2\Psi(\lambda,z)$ is not completely monotone. Also, \eqref{eq:column-growth} implies $a_n(\lambda)\to\infty$. Since a probability Hausdorff moment sequence is bounded by one, the moment property fails.

Now suppose $\lambda\ge1/4$. By \eqref{eq:column-kernel-signs}, $h(\lambda,t)>0$ for every $t>0$. From \eqref{eq:column-first-derivative}, for every $j\ge0$ and $z\ge0$,
\[
 (-1)^j\partial_z^{j+1}\Psi(\lambda,z)=\int_0^\infty t^j e^{-(\lambda+z)t}h(\lambda,t)\,dt>0.
\]
Hence $\partial_z\Psi(\lambda,z)$ and $-\partial_z^2\Psi(\lambda,z)$ are strictly completely monotone in $z$ on $[0,\infty)$.

We integrate \eqref{eq:column-first-derivative} from $0$ to $z$ and interchange the integrals by Tonelli's theorem. Since $\Psi(\lambda,0)=0$, we obtain
\begin{equation}\label{eq:column-levy-increment}
 \Psi(\lambda,z)=\int_0^\infty(1-e^{-zt})\frac{e^{-\lambda t}h(\lambda,t)}t\,dt,\qquad z\ge0.
\end{equation}
The function $h(\lambda,t)$ is bounded near zero, and $e^{-\lambda t}h(\lambda,t)$ decays exponentially at infinity. Thus the integral is finite and continuous at zero. Since $\Psi(\lambda,z)\ge0$ and $\partial_z\Psi(\lambda,z)$ is completely monotone, $\Psi(\lambda,z)$ is a Bernstein function of $z$. The strict inequality $\partial_z^2\Psi(\lambda,z)<0$ shows that it is not linear. By Lemma~\ref{lem:bernstein-moments}, for every $\tau>0$ there is a probability measure $\sigma_{\lambda,\tau}$ on $[0,1]$ with infinitely many positive support points such that
\[
 e^{-\tau\Psi(\lambda,n)}=\int_0^1 t^n\,d\sigma_{\lambda,\tau}(t),
 \qquad n\ge0.
\]
Since $e^{-\tau\Psi(\lambda,n)}=a_n(\lambda)^\tau$, $\{a_n(\lambda)\}_{n\ge0}$ is an infinitely divisible probability Hausdorff moment sequence.
\end{proof}

\begin{theorem}\label{thm:real-reciprocal-tp}
For every real $\lambda\ge1/4$ and every real $\tau>0$, the Hankel matrix
\[
 [a_{i+j}(\lambda)^\tau]_{i,j\ge0}
\]
is strictly totally positive.
\end{theorem}

\begin{proof}
By Theorem~\ref{thm:threshold-moments}, the sequence $\{a_n(\lambda)^\tau\}_{n\ge0}$ has a representing measure with infinitely many positive support points. Apply \eqref{eq:moment-minors} to this measure.
\end{proof}

For the reciprocal coefficients, fix $\ell\ge1$ and $\tau>0$. By Theorem~\ref{thm:threshold-moments},
\[
 \frac{1}{d_\ell(\ell+n)^\tau}=\frac{1}{d_\ell(\ell)^\tau}\int_0^1\left(\frac{t}{2^\tau}\right)^n\,d\sigma_{\ell,\tau}(t),\qquad n\ge0.
\]
Multiply $\sigma_{\ell,\tau}$ by $d_\ell(\ell)^{-\tau}$ and make the change of variable $v=t/2^\tau$. The resulting finite positive measure is supported on $[0,2^{-\tau}]$ and has $n$-th moment $d_\ell(\ell+n)^{-\tau}$. Thus the reciprocal coefficients raised to the power $\tau$ form a Hausdorff moment sequence.

For comparison, strict negativity of $\partial_z^2\Psi(\ell,z)$ for $z\ge0$, together with \eqref{eq:difference-cube}, implies $d_\ell(\ell)d_\ell(\ell+2)-d_\ell(\ell+1)^2<0$. Hence the Hankel matrix $[d_\ell(\ell+i+j)]_{i,j\ge0}$ is neither positive semidefinite nor strictly totally positive.

\begin{corollary}\label{cor:endpoint-atom}
Let $\lambda\ge1/4$ and $\tau>0$, and let $\sigma_{\lambda,\tau}$ be a probability measure on $[0,1]$ whose $n$-th moment is $a_n(\lambda)^\tau$ for every $n\ge0$. Then
\[
 \sigma_{\lambda,\tau}(\{1\})=\begin{cases}c_*^\tau,&\lambda=1/4,\\0,&\lambda>1/4,\end{cases}
\]
where
\[
 c_*=2^{-1/4}\Gamma(5/4)\Gamma(3/4)D(1/4,1/4)\in(0,1).
\]
\end{corollary}

\begin{proof}
By \eqref{eq:column-growth},
\[
 a_n(\lambda)\sim 2^{-\lambda}\Gamma(\lambda+1)\Gamma(3/4)
 D(\lambda,\lambda)(n+\lambda)^{1/4-\lambda}.
\]
On $[0,1]$, $t^n$ converges pointwise to the indicator of $\{1\}$. Dominated convergence gives
\[
 \sigma_{\lambda,\tau}(\{1\})
 =\lim_{n\to\infty}a_n(\lambda)^\tau,
\]
and the asymptotic formula yields the asserted values. For $\lambda=1/4$, the measure constructed in Theorem~\ref{thm:threshold-moments} has infinitely many support points, so $c_*<1$.
\end{proof}

\section{Logarithmic differences and asymptotics}\label{sec:degree-differences}

For fixed $\ell$, we prove the sign inequalities for $\Delta_m^r\log d_\ell(m)$ in Theorem~\ref{thm:column-log-differences} and deduce their consequences for iterated ratios. We apply finite differences to \eqref{eq:column-first-derivative}; the signs of $h(\ell,t)$ in \eqref{eq:column-kernel-signs} allow us to treat $\ell=0$ and $\ell\ge1$ using the same formula. We then study $\Delta_m^r d_\ell(m)$ using the asymptotic growth of the coefficients.

\begin{proof}[\normalfont\bfseries Proof of Theorem~\ref{thm:column-log-differences}]
For integers $\ell\ge0$ and $m\ge\ell$, set $\lambda=\ell$ in \eqref{eq:column-first-derivative} and integrate with respect to $z$ over $[m-\ell,m+1-\ell]$. Since
\[
 \Psi(\ell,m+1-\ell)-\Psi(\ell,m-\ell)
 =\Delta_m\log d_\ell(m)-\log2,
\]
we obtain
\[
 \Delta_m\log d_\ell(m)=\log2+\int_0^\infty e^{-mt}\frac{1-e^{-t}}t h(\ell,t)\,dt.
\]
The integral is absolutely convergent, including at $m=\ell=0$. For $r\ge2$, taking $r-1$ further differences yields
\[
 (-1)^{r-1}\Delta_m^r\log d_\ell(m)=\int_0^\infty e^{-mt}\frac{(1-e^{-t})^r}t h(\ell,t)\,dt.
\]
When $\ell\ge1$, $h(\ell,t)$ is strictly positive, and we obtain \eqref{eq:positive-differences}. When $\ell=0$, it is strictly negative, which gives \eqref{eq:zero-differences}.

For $\ell\ge1$, the first integral identity implies $\Delta_m\log d_\ell(m)>\log2>0$. For $\ell=0$, formula \eqref{eq:zero-column} gives $d_0(m+1)/d_0(m)=2(m+3/4)/(m+1)>1$. Thus the first logarithmic differences are positive in this case as well.
\end{proof}

\begin{corollary}[{\cite[Conjectures~4.2 and~4.5]{Zhao}}]\label{cor:transposed-consequences}
The sequence $\{d_0(m)\}_{m\ge0}$ is strictly infinitely log-monotonic. For every integer $\ell\ge1$, $\{d_\ell(m+1)/d_\ell(m)\}_{m\ge\ell}$ is strictly infinitely log-monotonic. For every $\ell\ge1$, $\{c_m^{(\ell)}\}_{m\ge\ell+1}$ is strictly log-concave.
\end{corollary}

\begin{proof}
For a positive sequence $a$, iteration of the ratio operator gives
\[
 \Delta_m^2\log(\mathcal{R}^j a)_m=\Delta_m^{j+2}\log a_m,\qquad j\ge0.
\]
For $a_m=d_0(m)$, Theorem~\ref{thm:column-log-differences} with
$r=j+2$ implies $(-1)^j\Delta_m^2\log(\mathcal{R}^j a)_m>0$.
For $a_m=d_\ell(m+1)/d_\ell(m)$ with $\ell\ge1$, we have
$\log a_m=\Delta_m\log d_\ell(m)$, and hence
\[
 (-1)^j\Delta_m^2\log(\mathcal{R}^j a)_m
 =(-1)^j\Delta_m^{j+3}\log d_\ell(m)>0,
\]
where the last inequality follows from the same theorem with $r=j+3$.
Thus, in both cases, $\mathcal{R}^j a$ is strictly log-convex for even $j$
and strictly log-concave for odd $j$.

Finally,
\[
 \log\frac{(c_m^{(\ell)})^2}{c_{m-1}^{(\ell)}c_{m+1}^{(\ell)}}
 =-\Delta_m^4\log d_\ell(m-2)>0,\qquad m\ge\ell+2,
\]
which proves strict log-concavity of $\{c_m^{(\ell)}\}_{m\ge\ell+1}$.
\end{proof}

For ordinary differences, only the growth estimate is needed. Together with \eqref{eq:zero-column}, equation \eqref{eq:column-growth} implies $d_\ell(m)\sim 2^m m^{\ell-1/4}/(\ell!\Gamma(3/4))$ as $m\to\infty$.

\begin{corollary}[{\cite[Conjecture~4.4]{JiangWang}}]\label{cor:forward-differences}
For every fixed pair of integers $\ell,r\ge0$,
\[
\begin{aligned}
 &\Delta_m^r d_\ell(m)\sim
 \frac{2^m m^{\ell-1/4}}{\ell!\Gamma(3/4)}\qquad(m\to\infty),\\[6pt]
 &\lim_{m\to\infty}
 \frac{\Delta_m^r d_\ell(m+1)}{\Delta_m^r d_\ell(m)}=2.
\end{aligned}
\]
In particular, $\Delta_m^r d_\ell(m)>0$ for all sufficiently large $m$.
\end{corollary}

\begin{proof}
Let $f_m=d_\ell(m)$. For fixed $j\ge0$, $f_{m+j}/f_m\to2^j$. Thus
\[
 \frac{\Delta_m^rf_m}{f_m}=\sum_{j=0}^r(-1)^{r-j}\binom rj\frac{f_{m+j}}{f_m} \longrightarrow\sum_{j=0}^r(-1)^{r-j}\binom rj2^j=1.
\]
Hence $\Delta_m^rf_m\sim f_m$, from which eventual positivity and the consecutive ratio limit follow.
\end{proof}

\section{Normalized coefficient ratios}\label{sec:reverse}

In this section, we prove Theorems~\ref{thm:binomial-normalization} and~\ref{thm:binomial-total-positivity}. We construct moment representations for the normalized sequence $\{b_{\ell+1+n}/b_{\ell+1}\}_{n\ge0}$ and all its positive real powers. The proof also establishes strict infinite log-monotonicity of $\{b_m\}_{m\ge\ell+1}$. We begin with an auxiliary function $H(\ell,y)$ whose logarithm has a positive Laplace representation.

For an integer $\ell\ge0$, define
\[
 H(\ell,y)=\frac{y-\ell+1}{y-\ell}\frac{D(\ell,y-1)D(\ell,y+1)}{D(\ell,y)^2}, \qquad y>\ell+1/2.
\]
At integer arguments, $H(\ell,m)=\frac{m-\ell+1}{m-\ell}c_m^{(\ell)}$. Consequently,
\begin{equation}\label{eq:two-normalizations}
 \frac{c_m^{(\ell)}}{\binom m{\ell+1}}=\frac{(\ell+1)H(\ell,m)}{(m-\ell+1)\binom m\ell},\qquad m\ge\ell+1.
\end{equation}
The factor $(y-\ell+1)/(y-\ell)$ removes the dependence on $\ell$ from the binomial contribution, since
\[
 \frac{y-\ell+1}{y-\ell} \frac{\binom{y-1}{\ell}\binom{y+1}{\ell}}{\binom y\ell^{\!2}} =\frac{y+1}{y},\qquad y>\ell+\tfrac12.
\]
For every $\ell\ge0$, the negative logarithmic derivative of $D(\ell,y)/(2^y\binom y\ell)$ with respect to $y$ has the positive Laplace representation \eqref{eq:normalized-log-derivative}, with kernel $k(\ell,t)$.

\begin{proposition}\label{cor:normalized-ratio-cm}
For every integer $\ell\ge0$, the logarithm of $H(\ell,\ell+1+z)$ is strictly completely monotone in $z$ on $[0,\infty)$. Thus $(-1)^r\Delta_m^r\log H(\ell,m)>0$ for $r\ge0$ and $m\ge\ell+1$, and $H(\ell,m)>1$. For $m\ge\ell+2$, with $n=m-\ell$,
\begin{equation}\label{eq:sharp-transposed-bound}
 \frac{(c_m^{(\ell)})^2}{c_{m-1}^{(\ell)}c_{m+1}^{(\ell)}} <\frac{n^3(n+2)}{(n-1)(n+1)^3}.
\end{equation}
The left side is also greater than one when $\ell\ge1$.
\end{proposition}

\begin{proof}
Fix $\ell\ge0$ and $y>\ell+1/2$. In \eqref{eq:normalized-log-derivative}, replace $y$ by $v$ and integrate over $[y-1,y]$ and $[y,y+1]$, as in \cite[eq.~(4.8)]{XieZhangUnified}. Subtract the integral over $[y,y+1]$ from the integral over $[y-1,y]$. Using the binomial identity above, we obtain
\[
 \log H(\ell,y)=\log\frac{y+1}{y}+\int_0^\infty e^{-yt}\frac{2\cosh t-2}{t}k(\ell,t)\,dt.
\]
Frullani's formula reads
\[
 \log\frac{y+1}{y}=\int_0^\infty e^{-yt}\frac{1-e^{-t}}t\,dt.
\]
Combining the two integrals and substituting $y=\ell+1+z$, we obtain
\[
 \log H(\ell,\ell+1+z)=\int_0^\infty e^{-zt}\widehat k(\ell,t)\,dt,
\]
where
\[
 \widehat k(\ell,t)=\frac{e^{-(\ell+1)t}}t[1-e^{-t}+(2\cosh t-2)k(\ell,t)]>0.
\]
The kernel $\widehat k(\ell,t)$ is bounded near zero and is $O(e^{-(\ell+3/4)t}/t)$ at infinity. Hence we may differentiate under the integral sign to every order for $z>-1/2$. Since $\widehat k(\ell,t)>0$, $\log H(\ell,\ell+1+z)$ is strictly completely monotone in $z$.

The strict signs of the finite differences follow from \eqref{eq:laplace-difference}. The case $r=0$ implies $H(\ell,m)>1$. For $m\ge\ell+2$, take $r=2$ at $m-1$. Using $H(\ell,m)=(n+1)c_m^{(\ell)}/n$, we obtain
\begin{align*}
 \frac{(c_m^{(\ell)})^2}{c_{m-1}^{(\ell)}c_{m+1}^{(\ell)}}
 &=\frac{n^3(n+2)}{(n-1)(n+1)^3}
 \frac{(H(\ell,m))^2}{H(\ell,m-1)H(\ell,m+1)}\\
 &<\frac{n^3(n+2)}{(n-1)(n+1)^3}.
\end{align*}
For $\ell\ge1$, Corollary~\ref{cor:transposed-consequences} gives the lower bound of one.
\end{proof}

\begin{proof}[\normalfont\bfseries Proof of Theorems~\ref{thm:binomial-normalization} and~\ref{thm:binomial-total-positivity}]
Fix $\ell\ge0$. By \eqref{eq:two-normalizations}, the function
\[
 J(\ell,y)=\log((\ell+1)!)+\log H(\ell,y)-\log(y-\ell+1)-\sum_{j=0}^{\ell-1}\log(y-j),\qquad y\ge\ell+1,
\]
interpolates $\log b_m$. The sum is empty when $\ell=0$. For $k\ge1$ and $y>a$,
\[
 (-1)^k\frac{d^k}{dy^k}[-\log(y-a)]
 =\frac{(k-1)!}{(y-a)^k}>0.
\]
Together with Proposition~\ref{cor:normalized-ratio-cm}, this implies
\[
 (-1)^k\partial_y^k J(\ell,y)>0,
 \qquad k\ge1,\quad y\ge\ell+1.
\]

For the moment representation, define
\[
 g(z)=J(\ell,\ell+1)-J(\ell,\ell+1+z),\qquad z\ge0.
\] Since $J(\ell,y)$ is strictly decreasing in $y$, $g(0)=0$ and $g(z)>0$ for $z>0$. For every integer $k\ge0$,
\[
 (-1)^k g^{(k+1)}(z)
 =(-1)^{k+1}\partial_y^{k+1}J(\ell,\ell+1+z)>0.
\] Thus $g'$ is completely monotone. Moreover, $g''<0$, so $g$ is a nonlinear Bernstein function. By Lemma~\ref{lem:bernstein-moments}, for every $\tau>0$ there is a probability measure $\mu_{\ell,\tau}$ on $[0,1]$ with infinitely many support points in $(0,1]$ such that
\[
 \left(\frac{b_{\ell+1+n}}{b_{\ell+1}}\right)^\tau=e^{-\tau g(n)}=\int_0^1 t^n\,d\mu_{\ell,\tau}(t),\qquad n\ge0.
\]
Since $\tau>0$ is arbitrary, $\{b_{\ell+1+n}/b_{\ell+1}\}_{n\ge0}$ is an infinitely divisible probability Hausdorff moment sequence. The Hankel matrix of these moments is strictly totally positive. Multiplying $\mu_{\ell,\tau}$ by $b_{\ell+1}^{\tau}$ produces a finite positive measure on $[0,1]$ whose $n$-th moment is $b_{\ell+1+n}^\tau$. Each minor of order $r$ is multiplied by the positive factor $b_{\ell+1}^{r\tau}$, so strict total positivity is preserved.

For strict infinite log-monotonicity, apply \eqref{eq:difference-cube} to the derivatives of $J$:
\[
 (-1)^j\Delta_m^2\log(\mathcal{R}^j b)_m=(-1)^j\Delta_m^{j+2}J(\ell,m)>0,\qquad j\ge0.
\]
Thus $\{b_m\}_{m\ge\ell+1}$ is strictly infinitely log-monotonic. In particular, $b_m^2<b_{m-1}b_{m+1}$ for $m\ge\ell+2$. Substituting the definition of $b_m$ and simplifying the binomial coefficients gives
\[
 \frac{(c_m^{(\ell)})^2}{c_{m-1}^{(\ell)}c_{m+1}^{(\ell)}} <\frac{\binom m{\ell+1}^{\!2}}{\binom{m-1}{\ell+1}\binom{m+1}{\ell+1}} =\frac{m(m-\ell)}{(m+1)(m-\ell-1)}.
\]
\end{proof}

The positive Laplace representation of $\log H(\ell,y)$ in $y$ provides a moment representation for every positive power of $H(\ell,m)$.

\begin{corollary}\label{cor:normalized-ratio-moments}
For every integer $\ell\ge0$ and every $p>0$, there is a finite positive measure $\sigma_{\ell,p}$ on $[0,1]$ such that
\[
 (H(\ell,\ell+1+n))^p=\int_0^1 z^n\,d\sigma_{\ell,p}(z),\qquad n\ge0.
\]
Its total mass is $(H(\ell,\ell+1))^p$, and the Hankel matrix
\[
 \left[\left(H(\ell,\ell+1+i+j)\right)^p\right]_{i,j\ge0}
\]
is strictly totally positive. For every $p>0$, the sequence
\[
 \left\{\left(\frac{H(\ell,\ell+1+n)}{H(\ell,\ell+1)}\right)^p\right\}_{n\ge0}
\]
is a probability Hausdorff moment sequence.
\end{corollary}

\begin{proof}
For $x\ge0$, let $\Lambda(x)=\log H(\ell,\ell+1+x)$. By Proposition~\ref{cor:normalized-ratio-cm}, $\Lambda$ is strictly completely monotone. The function $\Lambda(0)-\Lambda(x)$ is therefore nonnegative and vanishes at zero. Its derivative $-\Lambda'$ is completely monotone, and its second derivative $-\Lambda''$ is strictly negative. Thus it is a nonlinear Bernstein function. Applying Lemma~\ref{lem:bernstein-moments} with $\tau=p$, we obtain a probability measure $\widehat\sigma_{\ell,p}$ such that
\[
 \int_0^1 z^n\,d\widehat\sigma_{\ell,p}(z)=e^{p(\Lambda(n)-\Lambda(0))},\qquad n\ge0,
\]
and the Hankel matrix of these moments is strictly totally positive. Set $\sigma_{\ell,p}=e^{p\Lambda(0)}\widehat\sigma_{\ell,p}$. The stated mass and moment identities follow. Every minor of order $k$ is multiplied by $e^{kp\Lambda(0)}>0$, so strict total positivity is preserved.
\end{proof}

\section*{Declaration of generative AI}
During the preparation of this manuscript, the authors used generative AI tools
for assistance with language revision, manuscript organization, and the
development and checking of mathematical arguments and code. The authors reviewed and edited the resulting material and take full
responsibility for the content of the manuscript.

\end{document}